\documentclass[11pt]{article}

\usepackage[a4paper, left=3.2cm,top=3.2cm]{geometry}
\usepackage{bold-extra}

\usepackage{graphicx}
\usepackage{mathtools}
\usepackage{psfrag}
\usepackage{amssymb}
\usepackage{amsmath}

\usepackage{authblk}
\usepackage{siunitx}
\usepackage[onehalfspacing]{setspace}

\usepackage{amsthm}
\usepackage[hyphens]{url}
\usepackage{hyperref}\usepackage{todonotes}

\usepackage{algorithm}
\usepackage{algorithmic}
\usepackage{subfig}

\usepackage{xspace}

\usepackage{enumitem}
\setitemize{label=\scriptsize{$\blacksquare$},topsep=0em}
\setenumerate{label=(\alph*), topsep=0em}

\newtheorem{theorem}{Theorem}
\newtheorem{lemma}[theorem]{Lemma}
\newtheorem{proposition}[theorem]{Proposition}

\newtheorem{remark}[theorem]{Remark}

\theoremstyle{definition}
\newtheorem{condition}[theorem]{Condition}

\newcommand{\R}{\mathbb R}

\newcommand{\N}{\mathbb N}

\newcommand{\edot}{\,\cdot\,}

\newcommand{\e}{w}
\newcommand{\en}{E}

\DeclareMathOperator{\Io}{Id}
\DeclareMathOperator{\WF}{WF}

\newcommand{\eps}{\epsilon}

\newcommand{\la}{\lambda}

\newcommand{\diam}{\operatorname{diam}}
\newcommand{\dist}{\operatorname{dist}}

\newcommand\abs[1]{\left\vert#1\right\vert}
\newcommand\sabs[1]{{\lvert#1\rvert}}
\newcommand\norm[1]{{\left\Vert#1\right\Vert}}
\newcommand\snorm[1]{\Vert#1\Vert}

\newcommand\set[1]{{\left\{#1\right\}}}

\newcommand\inner[2]{[#1] \cdot [#2] }
\newcommand\sinner[2]{#1 \cdot #2}

\makeatletter
\newcommand*\bigcdot{\mathpalette\bigcdot@{.6}}
\newcommand*\bigcdot@[2]{\mathbin{\vcenter{\hbox{\scalebox{#2}{$\m@th#1\bullet$}}}}}
\makeatother

\newcommand{\Po}{\mathbf  P}                         
\newcommand{\Ko}{\mathbf  K}                        
                        
\newcommand{\Eo}{\mathbf  E}                         
                         
\newcommand{\timereversal}{\mathbf{W}^\sharp}    
\newcommand{\mtimereversal}{\mathbf{W}^\sharp}    

\newcommand{\cutop}{\boldsymbol \Psi}    
\newcommand{\cutoff}{ \Psi}    

\newcommand{\Wo}{\mathbf W}                         

\usepackage{soul}
\usepackage{xcolor}
\colorlet{lred}{red!40}
\colorlet{lgreen}{green!40}
\colorlet{lblue}{blue!40}

\numberwithin{equation}{section}
\numberwithin{figure}{section}
\numberwithin{theorem}{section}

\newcommand{\fnum}{{\tt f}}

\newcommand{\gnum}{{\tt{g}}}
\newcommand{\Wnum}{\boldsymbol{\tt{W}}}
\newcommand{\Rnum}{\boldsymbol{\tt{W}}^\sharp}

\newcommand{\Enum}{\boldsymbol{\tt{E}}}
\newcommand{\Pnum}{\boldsymbol{\tt{P}}}

\newcommand{\Om}{\Omega}

\allowdisplaybreaks

\title{Analysis for Full Field Photoacoustic Tomography with Variable Sound Speed}

\author{Linh V. Nguyen}

\affil{Department of Mathematics, University of Idaho\authorcr
              875 Perimeter Dr, Moscow, ID 83844, USA\authorcr
               E-mail: \tt{lnguyen@uidaho.edu}             }

\author{Markus Haltmeier}

\affil{Department of Mathematics, University of Innsbruck\authorcr
Technikerstrasse 13, 6020 Innsbruck, Austria
 \authorcr E-mail:  \texttt{markus.haltmeier@uibk.ac.at}
 }

\author{Richard~Kowar}

\affil{Department of Mathematics, University of Innsbruck\authorcr
Technikerstrasse 13, 6020 Innsbruck, Austria
 \authorcr E-mail:  \texttt{richard.kowar@uibk.ac.at}
 }

\author{Ngoc Do}

\affil{Department of Mathematics, Missouri State University \authorcr
Springfield, Missouri, USA\authorcr
 E-mail: \tt{ngocdo@MissouriState.edu} }

\begin{document}

\maketitle

\begin{abstract}
Photoacoustic tomography (PAT) is a non-invasive imaging modality that requires recovering the initial data of the wave equation from certain measurements of the solution outside the object. In the standard PAT measurement setup, the used data consist of time-dependent signals measured on an observation surface.  In contrast, the measured data from the recently invented full-field detection technique provide the solution of the wave equation on a  spatial domain at a single  instant in time. While reconstruction using classical PAT data has been extensively studied, not much is known for the full field PAT problem.    In this paper, we build   mathematical foundations of the latter problem for variable sound speed and settle its uniqueness and stability. Moreover, we introduce an exact inversion method using time-reversal and study its convergence. Our results demonstrate the suitability of both the full field approach and the proposed time-reversal technique for high resolution photoacoustic imaging.  


\end{abstract}

\section{Introduction} \label{sec:intro}

Consider the following initial value problem for wave equation  for an inhomogeneous isotropic medium 
\begin{equation}\label{forward}
\left\{ \begin{aligned}
     &\partial_t^2 p(x,t) - c^2(x)\,\Delta p (x,t) = 0   && \mathrm{for}\,  (x,t) \in \R^n\times (0,\infty)  
     \\ 
     &p(x,0) = f(x)  \qquad &&  \mathrm{for}\,  x \in \R^n
     \\
     &\partial_t p(x,0) = 0  \qquad &&  \mathrm{for}\,  x \in \R^n  \,.  
\end{aligned}  
\right.
\end{equation}
Here  $c \in C^\infty(\R^n)$ denotes the sound speed
and $f\in H_0^1(\R^n)$ the  initial data that is supported inside a bounded domain 
$\Omega\subseteq \R^n$ with Lipschitz boundary. We assume that the sound speed is positive everywhere and constant on the complement $\Omega^c \coloneqq \R^n \setminus \Omega$ of $\Omega$.  After rescaling we assume $c |_{\Omega^c} = 1$.  We refer to the solution  $p \colon \R^n \times [0, \infty) \to \R$ of \eqref{forward} as acoustic pressure field and $f$ as the initial pressure. 

Recall that $f \colon \Omega \to \R$  is an element of the Sobolev space $H^1(\Omega)$ if it is Lebesgue measurable and  $\norm{ f }^2_{H^1(\Omega)} \coloneqq \int_{\Omega} \, |\nabla f(x)|^2\,\mathrm{d} x + \int_{\Omega} \, | f (x)|^2\,\mathrm{d} x$ is finite.
Moreover, $H_0^1(\Omega)$ consists of all elements in $H^1(\Omega)$ that are supported inside $\Omega$. The space $H_0^1(\Omega)$ is equipped with the norm
\begin{equation*}
\| f \|_{H_0^1(\Omega)} \coloneqq \sqrt{ \int_{\Omega} \, |\nabla f(x)|^2\,\mathrm{d} x } \,,
\end{equation*}
which is  equivalent to $\norm{\, \cdot \, }_{H^1(\Omega)}$ when restricted  to  $H_0^1(\Omega)$.
We note that each $f\in H_0^1(\Omega)$ can be extended to a function of $H_0^1(\R^n)$ using the value zero on $\Omega^c$, which is tacitly  done in this paper. 

\subsection*{Full field photoacoustic tomography}

The aim of photoacoustic tomography (PAT) is to recover the initial pressure from certain observations of the  acoustic pressure field  made outside of $\Omega$.  In standard PAT, the  data  is given by the restricted  pressure $p | _{S \times [0,T]}$, where $S \subseteq  \R^n$ is an $(n-1)$-dimensional observation surface \cite{oraevsky1998optoacoustic,xu2006photoacoustic,finch2004determining,kuchment2008mathematics,beard2011biomedical,kuchment2013radon,haltmeier2017analysis,poudel2019survey}. Opposed to that, in full field PAT  introduced in  \cite{nuster2014high,nuster2010full}, the data provide the acoustic pressure  only for a single and fixed time  $T$ but  on an $n$-dimensional measurement domain.  

To be more specific, for  given $T>0$, we define the following two operators
\begin{align}\label{defAoe}
     & \Wo_T  \colon H_0^1(\Omega) \to H_0^1(\R^n) \colon  f\mapsto p(\,\cdot\,,T)  \\
     & \Wo_{T, \Omega} \colon H_0^1(\Omega) \to H^1(\bar \Omega^c)  \colon f \mapsto p(\,\cdot\,,T)|_{\bar \Omega^c}\,,
\end{align}
where $p$ is the solution of  \eqref{forward}. We refer to $\Wo_T$ as the complete single time wave transform  and to $\Wo_{T, \Omega}$ as the exterior single time wave transform. 
Full field PAT provides approximations of $\Wo_{T, \Omega} f$ from which the aim is to recover     
approximations to the initial  pressure $f$.
In \cite{zangerl2019full} it  is outlined  how actual full field PAT data can be reduced to  $\Wo_{T, \Omega} f$.

In this paper we prove uniqueness and stability of inverting  $\Wo_{T, \Omega}$ and derive an exact inversion procedure.  

\subsection*{Related work}

For the standard PAT problem there is a vast literature on various  practical and theoretical  aspects (see, for example,   \cite{xu2006photoacoustic,kuchment2008mathematics,beard2011biomedical,kuchment2013radon,haltmeier2017analysis,poudel2019survey}). In that context, the time-reversal method has been studied  intensively  \cite{finch2004determining,hristova2008reconstruction,stefanov2009thermoacoustic,stefanov2011thermoacoustic}. However, to the best of our knowledge, the time-reversal method has not developed  for PAT with full field data.   

Only few works exist \cite{nuster2010full,nuster2014high,zangerl2019full,haltmeier2019photoacoustic}  on the full field inversion problem. The work \cite{nuster2010full} considers  constant speed of sound and the problem is reduced to the inversion of the Radon transform. The work \cite{zangerl2019full} deals with non-constant speed and uses the standard Landweber iterative method. However, the article uses the data in the whole space, not the exterior data as we consider here. In the proceeding \cite{haltmeier2019photoacoustic}, variational regularization  is used with exterior data. Neither uniqueness nor stability has been proven there.   In the present article, for the first time, we  prove uniqueness and stability for  inverting $\Wo_{T, \Omega}$. Moreover,  we propose and analyze an iterative time-reversal procedure for its inversion.    

\section{Uniqueness and stability}
\label{sec:theory}

Let $\R^n$ be equipped with the metric  $c^{-2}(x)\,dx^2$. We denote by $\diam(\Omega)$ the diameter of $\Omega$, defined as the longest distance between any two points inside $\bar \Omega$ with respect to the metric $(g_{i,j})$.  We recall that $T>0$ is  a fixed observation time and  $\Omega \subseteq \R^n$ a domain with Lipschitz boundary.

\subsection{Uniqueness of reconstruction} Our first aim  is to prove the injectivity of  $\Wo_{T, \Omega}$, which implies that the full field PAT problem is uniquely solvable.  For that purpose we start by recalling  a uniqueness result for the wave equation due to Stefanov and Uhlmann \cite{stefanov2009thermoacoustic}.
\begin{lemma}\label{lemma:1} 
Let $f\in H_0^1(\R^n)$ and suppose $T > \diam (\Omega) / 2$. If the   solution $p$ of \eqref{forward}  satisfies  $p( \edot ,T)|_{\Omega^c} = 0$ and $(\partial_t p)(\edot,T)|_{\Omega^c} = 0$, then  $f = 0$.  
\end{lemma}
Denote by $B_{R} \subseteq \R^n$ the ball of radius $R > 0$ in the Euclidean metric of $\R^n$.
\begin{lemma}\label{lem:2} 
For $\eps >0$ and $h\in H_0^1(\R^n)$, let  $p \in C( [0,T], H^1(B_{T + \epsilon}))$ satisfy 
\begin{equation} \label{eq:wave:lemma}
\left\{ 
\begin{aligned}
     &\partial_t^2 u(x,t) - \Delta u (x,t) = 0 && \mathrm{for}\,  (x,t) \in B_{T+\epsilon} \times [0,T] \\ 
     &u(x,0) = h(x)  && \mathrm{for}\,  x \in B_{T+\epsilon}\\
     &\partial_t u(x,0) = 0 &&\mathrm{for}\,  x \in B_{T+\epsilon}\,.
\end{aligned}
\right. 
\end{equation}
Then  $h(x) =0$ for $x \in B_T$ and $u(x,T) = 0$ for $x \in B_\epsilon$ implies  $h(x) = 0$ for $x  \in B_{T+\epsilon}$.
\end{lemma}

\begin{proof}
For $u$ satisfying the Euler-Poisson-Darboux equation with initial data $(f,0)$ instead of the wave equation \eqref{eq:wave:lemma}, the result was proven in \cite{agranovsky2011support,nguyen2013spherical}. The proof of the current situation is similar to \cite[Theorem 2.1]{nguyen2013spherical} and  is therefore omitted. 
\end{proof}

In the following for  any $a>0$  we write  
\begin{align*}
	\Omega^{(1)}_a &\coloneqq  \set{ x \in \R^n \mid  \dist(x,\Omega) \leq a } \\
	 \Omega^{(2)}_a &\coloneqq \set{ x \in \R^n \mid  \dist(x,\Omega) \geq a } \,.
\end{align*}
Clearly,   for $f \in H_0^1(\Omega)$ we have $\Wo_{T, \Omega} f  \in H_0^1(\Omega^{(1)}_T)$. 

\begin{lemma} \label{lem:wave}
Let $\Omega$ be convex, $h \in H^1_0(\R^n)$ and suppose $u$ satisfies
\begin{equation*} 
\left\{ 
\begin{aligned}
     &\partial_t^2 u(x,t) - \Delta u (x,t) = 0  && \mathrm{for}\,  (x,t) \in  \Omega^{c} \times (0,\infty) \\ 
     &u(x,0) = h(x) && \mathrm{for}\,  x \in  \Omega^{c} \\ 
     &\partial_t u(x,0) = 0 && \mathrm{for}\,  x \in  \Omega^{c}   \,. 
\end{aligned}
\right. 
\end{equation*}
Then $u(x,T) = 0$ for all $x \in \Omega^{(2)}_T$ implies $h(x) =0$ for all $x \in  \Omega^{c}$.
\end{lemma}

\begin{proof}
Using Lemma~\ref{lem:2}, the proof  follows the lines of \cite[Proof of Theorem~3]{agranovsky2011support}  and for the sake of brevity is omitted. 
\end{proof}

Here is our main uniqueness result.

\begin{theorem}[Main injectivity result]\label{T:injectivity} 
If  $T > \diam (\Omega) / 2$, then the exterior single time wave transform  $\Wo_{T, \Omega} \colon H_0^1(\Omega) \to H^1(\bar \Omega^c)$ is injective.  In particular, the equation $\Wo_{T, \Omega} f = g$ has at most one  solution in $H_0^1(\Omega)$ for $g \in H^1(\bar \Omega^c)$.     
\end{theorem}

\begin{proof}  Suppose $f  \in H_0^1(\Omega)$ satisfies $\Wo_{T, \Omega} f = 0$ and denote by $p$ the solution of \eqref{forward}. By definition  we have $\Wo_{T, \Omega} f = p(\edot, T)|_{\bar \Omega^c}$ and thus $p(x,T) = 0$ and  $\Delta p(x,T) = 0$ for all $x \in \Omega^c$.  Define $u(x,t)  \coloneqq \partial_t p(x,T-t)$. Then $\partial_t u (x,t) = - \partial_t^2 p (x,T-t) = - \Delta p(x,T-t)$ in $\Om^c$. Consequently,  
\begin{equation*}
\left\{
\begin{aligned}
     &\partial_t^2 u(x,t) - \Delta u (x,t) = 0 && \mathrm{for}\,  (x,t) \in \Omega^c \times [0,T] \\ 
     &u(x,0) =  \partial_t p(x,T)  && \mathrm{for}\,  x \in \Omega^c  \\
     &\partial_t u(x,0) =0 && \mathrm{for}\,  x \in \Omega^c\,. 
\end{aligned}
\right.
\end{equation*}
Because   $u(x,T) = \partial_t p(x,0) =0$ in $\Omega^c$, Lemma~\ref{lem:wave}  shows  $\partial_t p(x,T) =0$ for all $x \in \Omega^c$.  Now application  of Lemma~\ref{lemma:1} gives $f=0$. 
\end{proof}

\subsection{Stability of inversion} Let us first recall some microlocal analysis for the solution of the wave equation; see for example \cite{stefanov2009thermoacoustic,treves1980introduction} for more details. Let  $\hat f(\xi) = \int_{\R^n} f(x) \, e^{-i x \cdot \xi} \, dx$ denote the Fourier transform of $f$.    Up to infinitely smooth error, the solution $p$ of \eqref{forward} can be written as 
\begin{equation} \label{E:para} 
p(x,t) = p^+(x,t) +  p^-(x,t) \coloneqq   \frac{1}{(2\pi)^{n}} \sum_{\sigma= \pm} \int_{\R^n} e^{i \phi_\sigma(x,\xi,t)} a_\sigma(x,\xi,t) \hat f(\xi) d\xi  \,.
\end{equation} 
Here, the phase functions $\phi_\pm(x,\xi,t)$ are positively homogenous of order $1$ in $\xi$ and solve the eikonal equations
\begin{equation*}
\left\{ \begin{aligned}
& \mp \partial_t \phi_\pm(x,\xi,t) = c(x) \,  \abs{ \nabla_x \phi_\pm(x,\xi,t) }  \\
 &  \phi_\pm(x,\xi,0) = \sinner{x }{ \xi } \,.
 \end{aligned}
 \right.
\end{equation*}
The functions $a_\pm$ are classical amplitudes of order $0$ satisfying $a_\pm(x,\xi,0) = 1/2$. The  principal terms $a^{(0)}_\pm(x,\xi,t)$ satisfy $a^{(0)}_\pm(x,\xi,t)=1/2$ and the homogenous equations
\begin{equation} \label{E:homo} [(\partial_t \phi_\pm) \partial_t -  c^2  \,   \nabla \phi_\pm \cdot \nabla_x + C_\pm] a_\pm^{(0)} =0 \,,
\end{equation} 
where $C_\pm  :=  (\partial_t^2 - c^2 \Delta) \phi_\pm / 2 $. Geometrically, each singularity $(x,\xi) \in \WF(f)$ is propagated by $p_+$ in the phase space along the positive bi-characteristic $(\gamma_{x,\xi}(t), \gamma'_{x,\xi}(t))$, while propagated by $p_-$ along the negative bi-characteristic given by $(\gamma_{x,-\xi}(t), \gamma'_{x,-\xi}(t)) = (\gamma_{x,\xi}(-t), - \gamma'_{x,\xi}(-t))$.

We consider the following so-called non-trapping condition. 

\begin{condition}[Non-trapping condition]\label{basiccond}
We assume that there exists a time $T_0>0$ such that each 
geodesic curve intersects $\Omega$ with the length at most $T_0$. 
\end{condition}

It is worth noting that if Condition~\ref{basiccond} holds then $\diam (\Omega) \leq T_0$. 

\subsubsection{Time-reversal operator}
For  $h \in H^1(\R^n)$ consider the following time-reversed wave equation 
\begin{equation}\label{eq:timerev}
\left\{ \begin{aligned}
     &\partial_t^2 q(x,t) - c^2(x)\,\Delta q (x,t) = 0 &&  \mathrm{for}\, (x,t) \in \R^n\times (0,T) 
     \\ 
     &q(x,T) =  h(x)  && \mathrm{for}\,  x \in \R^n 
      \\  
     &\partial_t q(x,T) = 0 && \mathrm{for}\,  x \in \R^n \, .
\end{aligned}
\right. 
\end{equation}
We define the time-reversal operator
\begin{equation} \label{defLambda}
 \timereversal_T  \colon H^1(\R^n) \to H^1(\Omega)  \colon h \to  q(\edot,0)|_{\Omega} \,,
\end{equation} 
where $q$ is the solution of  \eqref{eq:timerev}.
For a function  $\cutoff \in C_0^\infty(\R^n)$ denote by  $\cutop$ the pointwise multiplication operator  $f \mapsto  \cutoff f$. 

\begin{proposition} \label{prop:symbol} 
Let  $T > T_0/ 2 $, suppose $\cutoff \in C_0^\infty(\R^n)$ and set $x_+(x,\xi) \coloneqq \gamma_{x,\xi}(\pm T)$. Then $\timereversal_T \cutop \Wo_T \colon  H_0^1(\Omega) \to H^1(\Omega)$ is a pseudo-differential operator of order zero with principal symbol
\begin{equation*}
	\sigma(x,\xi) = \frac{1}{4} \Bigl[ \cutoff(x_+(x,\xi)) + \cutoff(x_-(x,\xi)) \Bigr] \,.
\end{equation*}
\end{proposition}

\begin{proof} 
Our key idea is the construction of the parametrix of time-reversed wave equation, the same spirit as \cite[Proof of Theorem 3]{stefanov2009thermoacoustic} but adapted to our context.   From (\ref{E:para}), up to smooth terms, we have $\timereversal_T \cutop \Wo_T =  \timereversal_T \cutop \Wo^{(+)}_T + \timereversal_T \cutop \Wo^{(-)}_T$ with
\begin{equation*}
	\Wo_T^{(\pm)} f (x) = p_\pm(x,T) = \frac{1}{(2\pi)^{n}} \int_{\R^n} e^{i \phi_\pm(x,\xi,T)} a_\pm(x,\xi,T) \hat f(\xi) d\xi \,.
\end{equation*}
It suffices to prove that $\timereversal_T \cutop \Wo^{(+)}_T$ is a pseudo-differential operator with principal symbol $\sigma_+(x_0,\xi_0) = \cutoff(x_+(x_0,\xi_0))/4$, for all $(x_0,\xi_0) \in T^* \Omega$.

Consider the parametrix of the time-reversed  wave equation \eqref{eq:timerev} with initial data $ h = \cutop \Wo^{(+)}_T f$, which can be written in the form 
\begin{multline*} q_+(x,t) = \frac{1}{2 (2\pi)^{n}}  \int_{\R^n} e^{i \phi_+(x,\xi,t)} b(x,\xi,t) \hat f(\xi) d\xi  
\\+  \frac{1}{2 (2\pi)^{n}} \int_{\R^n} e^{i \phi_+(x,\xi,2T-t)} b(x,\xi,2T-t) \hat f(\xi) d\xi \,,
\end{multline*}
with  $b(x,\xi,T) = \cutoff(x) a_+(x,\xi,T)$. Let us note that the first summand  in $q_+$ is a modification of the (positive) forward solution $p_+$ which in the case that  $\cutoff = 1$ exactly equals $p_+ / 2 $. The second summand  is the time-reflection of the first part through the value $t = T$. This construction imposes zero velocity at $t=T$.  Indeed, it is easy to check that $q_+$ satisfies the initial conditions $q_+(x,T)=\cutop \Wo^{(+)}_T f$ and $(q_+)_t(x,T)= 0$. Therefore, from the definition of $\timereversal_T$,
\begin{multline} \label{E:para1} 
 \timereversal_T \cutop \Wo^{(+)}_T f   =  q_+(x,0) = \frac{1}{2 (2\pi)^{n}}  \int_{\R^n} e^{i \phi_+(x,\xi,0)} b(x,\xi,0) \hat f(\xi) d\xi \\ + \frac{1}{2 (2\pi)^{n}} \int_{\R^n} e^{i \phi_+(x,\xi,2T)} b(x,\xi,2T) \hat f(\xi) d\xi,
\end{multline}
up to infinitely smoothing terms.

Note that both the principal term $a_+^{(0)}(x,\xi,t)$ and the principal term $b^{(0)}(x,\xi,t)$ of $b(x,\xi,t)$ satisfy the homogeneous transport equation (\ref{E:homo}). Hence, their ratio on each bi-characteristic is constant. In particular,
\begin{equation*}
\frac{b^{(0)}(x_0,\xi_0,0)}{a^{(0)}_+(x_0,\xi_0,0)} = \frac{b^{(0)}(x_+(x_0,\xi_0),\xi_+(x_0,\xi_0),T)}{a^{(0)}_+(x_+(x_0,\xi_0),\xi_+(x_0,\xi_0),T)}= \cutoff(x_+(x_0,\xi_0)) \,.
\end{equation*}
 Let us consider (\ref{E:para1}). Since $\phi_+(0,x,\xi) = x \cdot \xi$, the first part on the right hand side is a pseudo-differential operator with principal symbol at $(x_0,\xi_0)$ equals  
 \begin{equation*}
 \frac{1}{2} b^{(0)}(x_0,\xi_0,0) = \frac{1}{2} a_+^{(0)}(x_0,\xi_0,0) \cutoff(x_+(x_0,\xi_0)) 
 =  \frac{1}{4} \cutoff(x_+(x_0,\xi_0)) \,.
 \end{equation*}
The second summand of (\ref{E:para1}) is a Fourier integral operator that translates the singularity of $f$ at $(\gamma_{x_0,\xi_0}(-2T),\gamma_{x_0,\xi_0}'(-2T))$ to $(x_0,\xi_0)$. From the condition $T > T_0 / 2 $, we have $\gamma_{x_0,\xi_0}(-2T) \in \Omega^c$. Therefore, $f=0$ near $\gamma_{x_0,\xi_0}(-2T)$, which implies the second part on the right hand side of (\ref{E:para1}) is infinitely smoothing. This concludes our proof.
\end{proof}

\subsubsection{The stability result}

The following theorem  provides the stability of solving the final time wave inversion problem.  

\begin{theorem}[Main stability result] \label{T:stability} Assume that $T > T_0  / 2 $, with $T_0$ as in Condition~\ref{basiccond}. Then, there exists a constant  $C = C(\Omega, T , c)>0$ such that 
\begin{equation}
	\forall f \in H_0^1(\Omega) \colon \quad 
	\|f \|_{H_0^1(\Omega)} \leq C \, \norm{ \Wo_{T, \Omega} f }_{H_0^1(\Omega^c)} \,.
 \end{equation}
\end{theorem}

\begin{proof} 
Since $T > T_0 /2 $, there exists $a>0$ such that for all $(x,\xi) \in T^* \Omega$ either $x_+(x,\xi)  \in \Omega_a^{(2)}$ or $x_-(x,\xi) \in \Omega_a^{(2)}$.
Let $0 \leq \cutoff \in C_0^\infty (\R^n)$ be such that $\cutoff \equiv 0$ on $\Omega$ and $\cutoff \equiv 1$ on $\Omega_a^{(2)}$. Then $\cutop \Wo_{T,\Om} = \cutop \Wo_T $ and thus  Proposition~\ref{prop:symbol} implies that $\timereversal_T \cutop  \Wo_{T,\Om} $ is a pseudo-differential operator with  principal symbol $(\cutoff(x_+(x,\xi))  + \cutoff(x_-(x,\xi)))/4 \geq 1/4$.
Therefore,
\begin{equation*}
\forall f \in H_0^1(\Omega)\colon \quad 
	\norm{f}_{H^1(\Omega)} \leq C_1 \, \bigl( \snorm{ \timereversal_T \cutop \Wo_{T, \Omega} f }_{H^1(\Omega)}  + \norm{f}_{L^2(\Omega)} \bigr) \,.
\end{equation*}
Because $\timereversal_T  \cutop \Wo_{T, \Omega} f$ is supported  inside $\Omega^{(1)}_{2T}$ for some constant $C_2 >0$ we have  
\begin{align*}
\|\timereversal_T \cutop \Wo_{T, \Omega} f\|_{H^1(\Omega)} 
& \leq 
\| \timereversal_T  \cutop \Wo_{T, \Omega} f\|_{H^1(\R^n)}
\\  
& \leq C_2 \, \|  \timereversal_T  \cutop \Wo_{T, \Omega} f\|_{H_0^1(\R^n)} 
\\&  
\leq C_2 \, \| \cutop \Wo_{T, \Omega} f\|_{H_0^1(\R^n)} \,. 
\end{align*} 
Above, the last inequality comes from the conservation of the energy $  \int_{\R^n}  \bigl[  c^{-2} \abs{\partial_t p (\edot,t)}^2 + \abs{\nabla p \,(\edot,t)}^2 \bigr] \, \mathrm{d} x$ for  \eqref{eq:timerev}. 
From the last two displayed equations we  conclude 
\begin{equation} \label{eq:stabilityaux}
\forall f \in H_0^1((\Omega))\colon \quad 
\|f \|_{H_0^1(\Omega)} \leq C_1\, \bigl( C_2 \, \|\Wo_{T, \Omega} f\|_{H_0^1(\Omega^c)} + \norm{f}_{L^2(\Omega)} \bigr)  \,.
\end{equation}

Since $\Wo_{T, \Omega}$ is injective, and the embedding $H_0^1(\Omega) \to L_2(\Omega) \colon f \mapsto f$ is compact, applying \cite[Proposition 5.3.1]{taylor1996partial} to \eqref{eq:stabilityaux} concludes  the proof. \end{proof}

Let us briefly discuss the condition that $T > T_0 / 2$ posed in Theorem \ref{T:stability}. It implies for any $(x,\xi) \in T^* \Omega$, at least either $x_+(x,\xi)=\gamma_{x,\xi}(T)$ or $x_-(x,\xi)= \gamma_{x,\xi}(-T)$ belongs to $\Omega^c$. That is, if $(x,\xi) \in \WF(f)$ then either $(x_+(x,\xi), \xi_+(x,\xi) \coloneqq \gamma'_{x,\xi}(t)) \in \WF(\Wo_{T, \Omega} f)$ or $(x_-(x,\xi), \xi_-(x,\xi) \coloneqq \gamma'_{x,-\xi}(t)) \in \WF(\Wo_{T, \Omega} f)$. We, hence, say that all the singularities of $f$ are observed by $\Wo_{T, \Omega} f$. Therefore, $T > T_0/2$ is called the  visibility condition. We will always assume it in our subsequent presentation.

\section{Iterative time-reversal} \label{sec:algorithm}
Consider the extension operator $\Eo_\Omega  \colon H^1(\Omega^c) \to H^1(\R^n)$ as follows. For any $g\in H^1(\Omega^c)$, $\Eo_\Omega(g)$ restricted to $\Omega$ is given by the solution of the Dirichlet problem
\begin{equation}\label{eq:laplace}
\left\{ \begin{aligned}
    &\Delta h = 0   && \text { in }   \Omega \\ 
     &h =  g|_{\partial \Omega} && \text { on }   \partial \Omega \,.
\end{aligned}
\right. 
\end{equation}
Here, $g |_{\partial \Omega} \in H^{1/2} (\partial \Omega) $ denotes the trace of $g \in H^1(\Omega^c)$ on $\partial \Omega$.  Note that the Dirichlet interior problem \eqref{eq:laplace} has a unique solution $h \in H^1(\Omega)$ (see, for example, \cite{mclean2000strongly}); therefore, noting that $\Eo_\Omega(g) =g$ on $\Omega^c$, $\Eo_\Omega(g) \in H^1(\R^n)$. For notational conveniences, we sometimes  use the short-hand notation $\bar g$ for $\Eo_\Omega(g)$.  We further define the orthogonal projection $\Po_\Omega  \colon  H^1(\Omega) \to H_0^1(\Omega) \colon g \mapsto  g - h$, where $h\in H^1(\Omega)$ is the solution of  \eqref{eq:laplace}.

Recall that our aim  is the inversion of the  restricted single time wave inversion operator   $\Wo_{T, \Omega} \colon  H_0^1(\Omega) \to H^1(\bar \Omega^c)$  defined  by  \eqref{defAoe}.  Our proposed inversion approach is  based on the \emph{modified time-reversal} operator  defined by    
\begin{equation*} 
	\mtimereversal_{T,\Omega}  \coloneqq    \Po_\Omega   \timereversal_T    \Eo_\Omega \colon  H^1(\Omega^c) \to H^1_0(\Omega) \,.
\end{equation*} 
The modified time-reversal operator  is itself the composition of harmonic extension $\Eo_\Omega$ to $\R^n$, time-reversal  $\timereversal_T$ defined by (\ref{defLambda}) and  projection $\Po_\Omega$ onto $H^1_0(\Omega)$.

\subsection{Contraction property}

In the case that  the sound speed is constant, space  dimension is odd and the  measurement time satisfies $T > T_0$, we have $\mtimereversal_{T,\Omega} =  \timereversal_T$. In particular, in the one dimensional space, the operator $2 \, \mtimereversal_{T,\Omega} $ is the exact inverse of $\Wo_{T,\Omega}$. In the general case this is not true.  Nevertheless, as the basis our approach, we will show  that  the error operator $ \Io -   \la \, \mtimereversal_{T,\Omega}   \Wo_{T,\Omega}  $ is non-expansive for $ \lambda = 2$  and  a contraction for $\la < 2$.    This will serve as the basis of the proposed iterative time-reversal procedure.  

Throughout the following we denote by  $\en_\e(t)  \coloneqq  \int_{\R^n}  \bigl[  c^{-2}(x) \abs{\partial_t \e (x,t)}^2 + \abs{\nabla \e \,(x,t)}^2 \bigr] \,\mathrm{d} x$ the energy associated to  a function $\e$ satisfying  the wave equation $\partial_t^2 \e  = c^2 \,\Delta \e $. 

\begin{theorem}[Contraction property of the error operator] \label{theo:main}
Suppose $T> T_0/2 $ and consider for any  $\lambda \in (0,2]$ the error operator 
\begin{equation*} 
	\Ko_{T,\Omega,\la} 
	\coloneqq \Io -   \la \, \mtimereversal_{T,\Omega}    \Wo_{T,\Omega}  
	\colon  H_0^1(\Omega) \to H_0^1(\Omega) \,.
\end{equation*} 
Then the following hold: 
\begin{enumerate}[label=(\alph*)] 
\item \label{ele1} $\Ko_{T,\Omega,2}$ satisfies $\forall f \in H^1_0 \setminus\set{0}  \colon \|\Ko_{T,\Omega,2} f\| < \norm{f}$. 
\item \label{ele2}   If $\lambda \in (0,2)$, then  $\|\Ko_{T,\Omega,\la}\| < 1$. 
\end{enumerate}
\end{theorem}

\begin{proof}  
\mbox {}\ref{ele1}: 
For $f \in H_0^1(\Omega)$, set $g \coloneqq \Wo_{T,\Omega} f$ and $\bar g \coloneqq \Eo_\Om g$. Moreover, let $p$ solve the forward problem~\eqref{forward} and let $q$ solve the time-reversal  problem~\eqref{eq:timerev} with $h = 2 \bar g$. Then $\e \coloneqq p - q$ satisfies the wave equation  $\partial_t^2 \e  = c^2 \,\Delta \e $  and the corresponding energies at times $0$ and  $T$ respectively satisfy  
\begin{align}
      \en_\e(0) &= 
      \int_{\R^n}  \left[ c^{-2}(x)\,|  \partial_t q (x,0)|^2
      + \sabs{ \nabla  q(x,0) - \nabla f(x) }^2 \right] \,\mathrm{d} x,  \nonumber 
\\   \label{eq:ene2}
 \en_\e(T) &= \int_{\R^n}  \left[ c^{-2}(x)\,|\partial_t p (x,T)|^2 + \sabs{2   \nabla \bar g (x) - \nabla g (x) }^2 \right] \,\mathrm{d} x  \,.
\end{align}
The trace extension  satsfies   $\bar g |_{\partial \Omega} = g |_{\partial \Omega} $ and $(\Delta \bar g) |_\Omega =0$. Therefore   
\begin{equation*}
       \int_{\Omega} (\sabs{2 \nabla \bar g - \nabla g }^2 - \abs{\nabla g}^2) \, \mathrm{d} x
          = 4  \int_{\Omega} \inner{\nabla \bar g  }{ \nabla ( \bar g- g) } \,\mathrm{d} x
           = 4 \int_{\Omega} \Delta \bar g \, ( \bar g  -g)  \, \mathrm{d} x
          = 0 \,.
\end{equation*} 
We obtain $ \int_{\Omega} \sabs{2 \nabla \bar g - \nabla g }^2 \, \mathrm{d} x= \int_{\Omega} \abs{\nabla g}^2 \, \mathrm{d} x$ and from \eqref{eq:ene2} we deduce $\en_\e(T) = \en_p(T) $.
With the  conservation of energy we have  $\en_p(0) =  \en_\e(0)  $ and therefore 
\begin{equation}\label{intrelf}
       \int_{\R^n} \abs{ \nabla f(x) }^2 \,\mathrm{d} x 
        = \int_{\R^n}  \left[ c^{-2}(x) \, \abs{ \partial_t q  (x,0)}^2 
                         + \abs{\nabla  q(x,0) - \nabla f(x) } ^2 \right] \,\mathrm{d} x \,,
\end{equation}
where we have used the explicit expressions for $\en_p(0) $ and $\en_\e(0) $ respectively.

With  $f^* \coloneqq 2  \mtimereversal_{T,\Omega}   \Wo_{T,\Omega} f $ the error operator satisfies  $\Ko_{T,\Omega,2} f = f - f^*$.  Moreover, writing  $q_0 \coloneqq  q(\cdot,0)|_\Omega$ we have $f^* = \Po_\Omega (q_0)$  and thus  $\Delta [ q_0 - f^*] = 0$ in  $\Omega$. From this we infer
 $ \int_\Omega  \inner{   \nabla q_0 -  \nabla f^* }{ \nabla f^* - \nabla f }  \, \mathrm{d} x
     = -\int_\Omega  \inner{ \Delta q_0 - \Delta  f^* }{ f^*  - f }  \, \mathrm{d} x \ 
     = 0$ and therefore 
\begin{equation} \label{eq:est2}
     \int_\Omega  \abs{ \nabla q_0 - \nabla f } ^2  \, \mathrm{d} x
     = \int_\Omega  \abs{ \nabla q_0 - \nabla f^* } ^2 \, \mathrm{d} x
     + 
     \int_\Omega  \abs{ \nabla f^* - \nabla f } ^2 \, \mathrm{d} x
     \geq  \int_\Omega  \abs{ \nabla f^* - \nabla f } ^2 \, \mathrm{d} x.
 \end{equation}
Together  with~(\ref{intrelf}) this implies 
$ \snorm{f}^2_{H_0^1(\Omega)}  \geq \norm{f - f^*}^2_{H_0^1(\Omega)} $ and therefore
$\snorm{ \Ko_{T,\Omega,2} f }= \|f-f^*\|_{H_0^1(\Omega)} \leq  \snorm{f}_{H_0^1(\Omega)} $.

In remains to show the strict inequality. To that end assume $ \|f-f^*\|_{H_0^1(\Omega)} =   \norm{f}_{H_0^1(\Omega)} $. From (\ref{intrelf}) and \eqref{eq:est2} we obtain
\begin{equation*}
       \int_\Omega  \abs{ \nabla q_0 - \nabla f } ^2  \, \mathrm{d} x 
        \geq  \int_{\R^n}  \left[ c^{-2}(x) \, \abs{ \partial_t q  (x,0)}^2 
                         + \abs{\nabla  q(x,0) - \nabla f(x) } ^2 \right] \,\mathrm{d} x \,,
\end{equation*}
 and therefore 
 $$ \int_{\R^n}  c^{-2}(x)\,|\partial_t q (x,0)|^2 \,\mathrm{d} x  
                           + \int_{\Omega^c}  |\nabla q(x,0)|^2 \,\mathrm{d} x  =0 \,.$$
In particular, $\partial_t q (\cdot,0)$  vanishes on $\R^n$ and $\nabla q(\cdot,0)$ vanishes on $\Omega^c$. Because $q(x,0)$ vanishes for $x \in \Omega^{(1)}_{2T}$, it follows that  $q(\cdot,0)$ vanishes on $\Omega^c$. 
Applying Lemma~\ref{lemma:1} for $u(\cdot,t) \coloneqq q(\cdot,T-t)$ yields $2 \bar g = q(\cdot,T) = u(\cdot,0) = 0$ on $\R^n$. In particular, $\Wo_{T, \Omega} f = 0$ on $\Omega^c$. From Theorem~\ref{T:injectivity}, we infer $f=0$  on $\R^n$, which concludes the proof.

\ref{ele2}:  
Let us first consider the case $\lambda=1$.  We have to show that there exists a  constant $L < 1$ such that $\snorm{ \Io - \Wo_{T, \Omega}^\sharp \Wo_{T, \Omega})   } \leq L $.
To that end, let $f \in H^1_0(\Omega)$, $p$ solve the forward model~\eqref{forward} with initial data $f$, $q$ solve the time-reversal problem \eqref{eq:timerev} with initial data $h = \Eo_\Omega \Wo_{T, \Omega} f $ and define the error term $\e \coloneqq q - p$.    
The error term satisfies the wave equation 
$\partial_t^2 \e - c^2 (x) \Delta \e = 0 $ in $\R^n\times (0,T)$ and its energy at time $T$ is given by    
\begin{align*}
     \en_\e(T)  &=  \int_{\R^n}  \bigl[  c^{-2}(x)\,|\partial_t \e (x,T)|^2 + |\nabla \e \,(x,T)|^2 \bigr] \,\mathrm{d} x
     \\ &= \int_{\R^n}  c^{-2}(x) \abs{ \partial_t p (x,T)}^2 \,\mathrm{d} x 
     + \int_{\R^n} \abs{\nabla \bar{g}(x) - \nabla g (x) }^2  \,\mathrm{d} x \,.
\end{align*}
Here for the second equality we used the conditions   $q(\edot,T) = \bar g \coloneqq \Eo_\Omega g$ and $\partial_t g(\edot ,T) = 0$ and the abbreviation $g  =  p( \edot ,T) $.

The second term  in the above equation displayed satisfies  
\begin{multline*}
\int_{\R^n}  \abs{\nabla \bar{g}(x) - \nabla g (x) }^2  \,\mathrm{d} x =
\\
     \begin{aligned}
      &=      \int_{\R^n}  \inner{\nabla (g(x) - \bar{g} (x)) }{ \nabla (g(x) - \bar{g} (x)) } \,\mathrm{d} x 
     \\ 
     &=  \int_{\R^n}  \inner{\nabla (g(x) - \bar{g} (x)) }{ \nabla (g(x) + \bar{g} (x)) } \,\mathrm{d} x  
     - 2 \int_{\R^n}  \inner{\nabla (g(x) - \bar{g} (x)) }{ \nabla \bar{g} (x)  } \,\mathrm{d} x 
     \\ 
     &=  \int_{\Omega}  \abs{\nabla g (x)}^2  \,\mathrm{d} x 
          - \int_{\Omega} \abs{ \nabla \bar g(x)}^2 \,\mathrm{d} x  
          - 2 \int_{\R^n}  \bigl( g(x) - \bar{g} (x) \bigr) \, \Delta \bar{g} (x)  \,\mathrm{d} x 
       \\ 
       &= \int_{\Omega}  \abs{\nabla g (x)}^2  \,\mathrm{d} x 
         - \int_{\Omega} \abs{\nabla \bar g(x)}^2  \,\mathrm{d} x 
         \\
         & \leq \int_{\Omega}  \abs{\nabla g (x)}^2  \,\mathrm{d} x \,.
\end{aligned}
\end{multline*}
As a consequence, we obtain 
\begin{equation*}
     \begin{aligned}
     \en_\e(T) 
        &\leq \int_{\R^n}  \bigl[ c^{-2}(x) \abs{ \partial_t p (x,T)}^2 
             + \abs{\nabla g(x) }^2 \bigr] \,\mathrm{d} x 
              - \int_{\Omega^c}  \abs{\nabla g (x)}^2  \,\mathrm{d} x \, 
              \\&= 
              \int_{\R^n}  \bigl[ c^{-2}(x) \abs{\partial_t p (x,T)}^2 
               + \abs{\nabla p(x,T)}^2 \bigr] \,\mathrm{d} x 
               - \| \Wo_{T, \Omega} f\|^2_{H_0^1(\Omega^c)}  
               \\&= \en_p(T)- \|\Wo_{T, \Omega} f\|^2_{H_0^1(\Omega^c)} \,.
\end{aligned}
\end{equation*}
Together with the conservation of energy and using the initial conditions $p(x,0) = f(x)$ and $\partial_t p(x,0) = 0$ this shows   
\begin{align*}
          \en_\e(0) +\|\Wo_{T, \Omega} f\|^2_{H_0^1(\Omega^c)}
          &=
          \en_\e(T) +\|\Wo_{T, \Omega} f\|^2_{H_0^1(\Omega^c)}  
          \\
          &\leq \en_p(T) = \en_p(0)
         = \norm{f}^2_{H_0^1(\Omega)}.
\end{align*}
Using that  $\en_\e(0)   = \int_{\R^n}  \left[ c^{-2}(x)\,|\partial_t q (x,0)|^2 + |\nabla q(x,0) - \nabla f(x)|^2 \right] \,\mathrm{d} x$  and applying Theorem~\ref{T:stability}  we obtain 
\begin{equation} \label{E:ineq} \int_\Omega \Big|\nabla q(x,0) - \nabla f(x)\Big|^2 \,\mathrm{d} x
  \leq \left( 1-\frac{1}{C^2} \right) \norm{f}^2_{H_0^1(\Omega)}.\end{equation}
The left hand side in the above equation can be estimated as  
\begin{align*}  
\int_\Omega \sabs{ \nabla q(x,0) - \nabla f(x)}^2 \,\mathrm{d} x  
& = 
\int_\Omega \sabs{ \nabla (q(\cdot,0) - \Po_\Omega (q(\cdot,0)) ) + \nabla (\Po_\Omega (q(\cdot,0)) - f) } ^2 \,\mathrm{d} x \\ 
& =\int_\Omega \sabs{ \nabla (q(\cdot,0) - \Po_\Omega (q(\cdot,0)))}^2 + \sabs{ \nabla (\Po_\Omega (q(\cdot,0)) - f) }^2 \,\mathrm{d} x 
\\& 
\geq \norm{\Po_\Omega (q(\cdot,0))  - f}^2_{H_0^1(\Omega)},
\end{align*} 
where we have used the fact that
$\int_\Omega   \inner{\nabla q(\cdot,0) - \nabla \Po_\Omega (q(\cdot,0))}{ \nabla \Po_\Omega (q(\cdot,0)) - \nabla  f}\,\mathrm{d} x = \int_\Omega \Delta [q(\cdot,0) - \Po_\Omega (q(\cdot,0))] \, (\Po_\Omega (q(\cdot,0)) - f)\,\mathrm{d} x =0$. From \ref{E:ineq}, we arrive at 
\begin{equation*} 
\|\Po_\Omega (q(\cdot,0))  - f\|^2_{H_0^1(\Omega)} = \| \Ko_1 f\|^2_{H_0^1(\Omega)} \leq  \left( 1-\frac{1}{C^2} \right) \norm{f}^2_{H_0^1(\Omega)} \,.
\end{equation*} 
This finishes the proof for the case $\lambda =1$. 

For  the general case note the identities   
\begin{equation*}
\Ko_{T,\Omega,\la}  =  
\begin{cases} 
(1 - \la )\,\Io + \la \Ko_{T,\Omega,1} & \text{  for } \lambda \in (0,1) 
\\
 (\la-1) \Ko_{T,\Omega,2} + (2-\la) \Ko_{T,\Omega,1} & \text{  for } \lambda \in (1,2) \,.
\end{cases}
\end{equation*}
Using the already verified estimates  $\snorm{\Ko_{T,\Omega,1}} < 1$ and $\snorm{\Ko_{T,\Omega,2}} \leq 1$, these equalities together with the  triangle inequality for the operator norm  show $\snorm{\Ko_{T,\Omega,\la}} <  1$ for all $\lambda \in (0,2)$. 
\end{proof}

\subsection{Neumann series solution}\label{sec:neumann}
According to Theorem~\ref{theo:main} the error operator  satisfies $\snorm{\Io -   \la \, \mtimereversal_{T,\Omega}    \Wo_{T,\Omega}}  <  1 $ for any $\lambda \in (0,2)$.  The  Neumann series $\sum_{j=0}^\infty (\Io -   \la \, \mtimereversal_{T,\Omega}\Wo_{T,\Omega})^j $, therefore, converges to $(\la \, \mtimereversal_{T,\Omega}\Wo_{T,\Omega})^{-1}$ with respect to the operator norm $\norm{\edot}$ in  $H^1_0(\Omega)$.   This results in the inversion formula 
\begin{equation}\label{NeumanRep}
   f =
         \sum_{j=0}^\infty ( \Io -   \la \, \mtimereversal_{T,\Omega}   \Wo_{T,\Omega})^j  \, (  \la \, \mtimereversal_{T,\Omega}   g)  \quad \text{ with } g = \Wo_{T, \Omega} f 
\end{equation}
valid for every initial data $ f  \in H^1_0(\Omega) $. 
Here $\mtimereversal_{T,\Omega} =   \Po_\Omega   \timereversal_T    \Eo_\Omega$ is the modified time-reversal operator formed by harmonic extension $\Eo_\Omega$ of the missing data, time-reversal $\timereversal_T$ defined by \eqref{eq:timerev} and projection $\Po_\Omega$ onto  $H^1_0$.  Inversion formula \eqref{NeumanRep} is  the  Neumann series solution for the inverse problem of full-field PAT. 

\begin{remark}[Iterative time-reversal algorithm]
The Neumann series is \eqref{NeumanRep} is the limit of its partial sums $f_k \coloneqq  \sum_{k=0}^j ( \Io -   \la \, \mtimereversal_{T,\Omega}  \Wo_{T,\Omega})^k    (  \la \, \mtimereversal_{T,\Omega}   g)$. These  partial sums satisfy the recursion 
\begin{equation} \label{eq:iter}
    \left\{
    \begin{aligned}
    f_0  &= \la \mtimereversal_{T,\Omega}  \,  g \\
    f_{j} &= f_{j-1}  - \la  \mtimereversal_{T,\Omega} ( \Wo_{T, \Omega} f_{j-1}  - g)  \,,
\end{aligned}
\right. 
\end{equation}
with $\mtimereversal_{T,\Omega} = \Po_\Omega   \timereversal_T    \Eo_\Omega$.
This is an  iterative algorithm producing a sequence $(f_j)_{j\in \N}$ converging to $f  =   \Wo^{-1}_{T,\Omega} g$  in   $H^1_0(\Omega) $.  We call  \eqref{eq:iter} iterative reversal algorithm for full field PAT.  The form  \eqref{eq:iter} will be used in the numerical solution. Because of the  contraction property  of the iteration $ \snorm{\Io -   \la \, \mtimereversal_{T,\Omega}  \Wo_{T,\Omega} } < 1 $  the iterative time-reversal reversal algorithm is linearly convergent. 
\end{remark}

We note that for  standard  PAT, the idea of using time-reversal was proposed in \cite{finch2004determining,burgholzer2007exact} for the case of constant sound speed, and in \cite{grun2008photoacoustic,hristova2008reconstruction} for non-constant sound speed. The Neumann series solution was first proposed in \cite{stefanov2009thermoacoustic} and further developed in \cite{stefanov2009thermoacoustic,stefanov2011thermoacoustic,tittelfitz2012thermoacoustic,homan2013multi,stefanov2015multiwave,nguyen2016dissipative,palacios2016reconstruction,katsnelson2018convergence,acosta2018thermoacoustic}.  Iterative reconstruction methods  for variable sound speed based on an adjoint wave equation have been studied in \cite{huang2013full,belhachmi2016direct,arridge2016adjoint,haltmeier2017analysis,javaherian2018continuous}.   
Uniqueness  and stability for standard PAT was studied in \cite{xu2004reconstructions,hristova2008reconstruction,stefanov2009thermoacoustic,stefanov2011thermoacoustic,nguyen2011singularities}, just to name a few.

\begin{figure}[htb!]
\centering 
\includegraphics[width=0.32\textwidth]{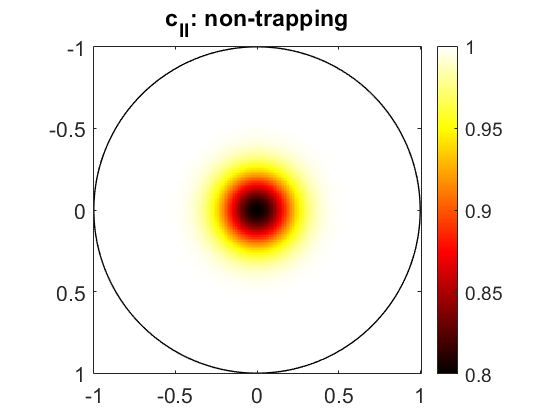}
\includegraphics[width=0.32\textwidth]{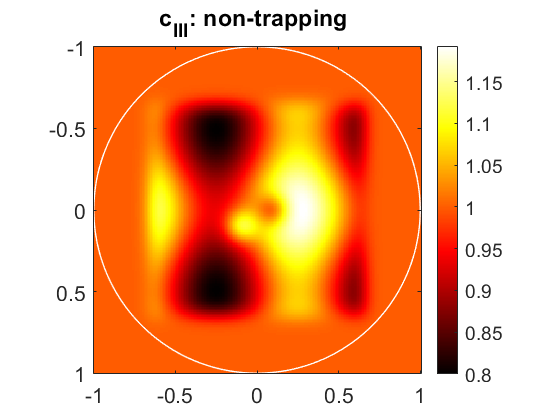}
\includegraphics[width=0.32\textwidth]{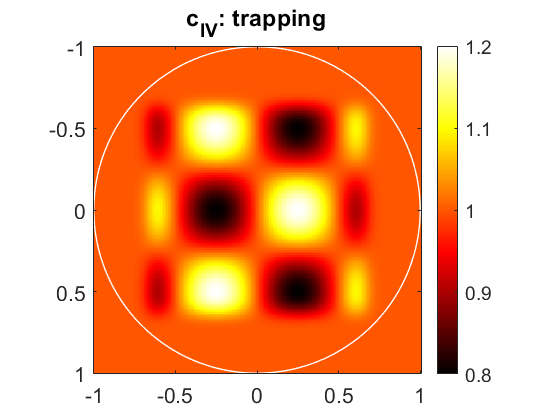} \\[0.2em]
\includegraphics[width=0.32\textwidth]{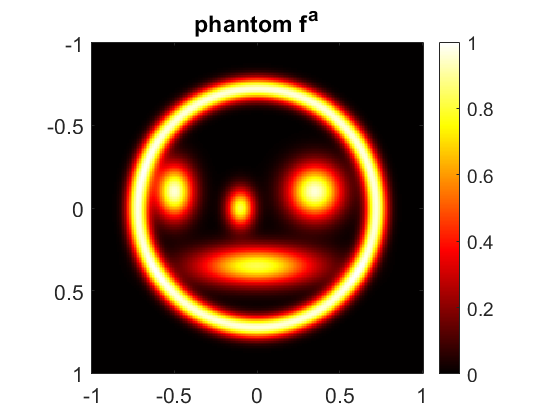}
\includegraphics[width=0.32\textwidth]{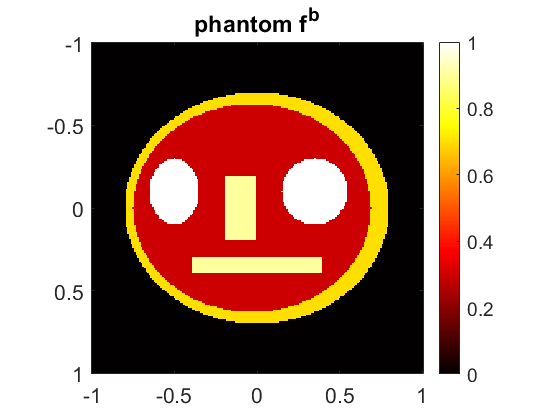}
\includegraphics[width=0.32\textwidth]{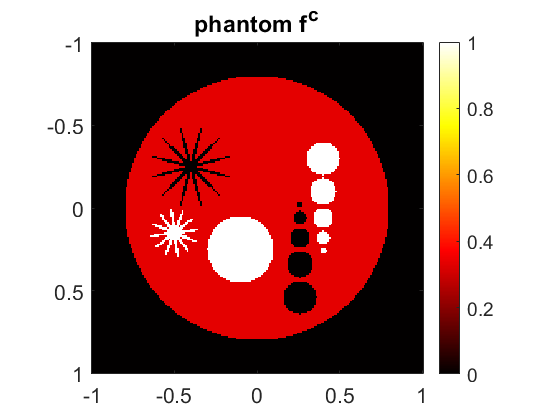}
\caption{\textbf{Sound speeds profiles and initial pressure distributions.}  Top: Three non-trapping ($c_{\rm II}$ and $c_{\rm III}$) and one trapping sound speed profile  $c_{\rm IV}$ that we employed in our simulations besides the constant speed of sound $c_{\rm I}=1$. The white and black circles visualize  the boundary of the imaging region which in our simulations is the unit disc. Bottom: A smooth phantom $f^a$ (left) and the two piecewise constant phantoms $f^b$ (middle) and $f^c$ (right) are employed in our numerical simulations.}
\label{fig:phantoms}
\end{figure}

\section{Numerical Simulations} \label{sec:numerics}

In this section, we present  some of our numerical studies for the exterior  single time wave transform. We consider the case of  two spatial dimensions, and take  $\Omega \subseteq \R^2  $ as the unit disc.   According to \eqref{NeumanRep} any function $f \in H^1_0 (\Omega)$ can be recovered from data $g = \Wo_{T, \Omega} f$ via the iterative time-reversal algorithm \eqref{eq:iter}.    The numerical realization is described in   the following subsection.  The numerical simulations were performed for each of the three sound speed profiles shown in the top row of Figure~\ref{fig:phantoms} and additionally for the constant sound speed $c_{\rm I} = 1$.   As phantom we used numerical approximations of one smooth and two piecewise constant functions which are visualized in the bottom row of Figure~\ref{fig:phantoms}. 

\subsection{Numerical implementation} \label{sec:realization}

In the numerical realization, any function $h \colon \R^2 \to \R$ is represented by a discrete vector $(h(x_{i}))_{i_1, i_2 = 0}^{N-1} \in \R^{N  \times N } $, where 
\begin{equation*}
	x_i = (-a,-a) + 2 i a/N \quad \mathrm{for}\, \quad  i  = (i_1, i_2)\in \set{0, \dots, N-1}^2 
\end{equation*}
are equidistant grid points in the square $[-a,a]^2$.  The discrete domain $I \subseteq   \set{0, \dots, N-1}^2$ (where the discrete initial pressure is contained in) is defined as the set of all indices $i$ with $x_i  \in \Omega$  and we set  $J  \coloneqq  \set{0, \dots, N-1}^2 \setminus I$. Following \cite{haltmeier2017iterative}, we define the discrete boundary of $I$ as the set of all elements $(i_1,i_2)  \in J$ for which at least one of the discrete neighbors $(i_1+1,i_2), (i_1-1,i_2), (i_1,i_2+1), (i_1,i_2-1)$ is  contained  in $I$. The discrete version of the initial data $f \in H^1_0(\Omega)$ is then an image $\fnum \in \R^I$ and the discrete version of the data $g \in L^2(\Omega^c)$  an image $\gnum \in \R^J$.

In the iterative time-reversal algorithm the forward transform $ \Wo_{T, \Omega}$ as well as each  factor  in the  modified time-reversal  $\mtimereversal_{T,\Omega} = \Po_\Omega   \timereversal_T    \Eo_\Omega$ are replaced by discrete approximations.  The discrete forward operator and the discrete time-reversal operator are  defined by 
\begin{align} \label{eq:fwdN}
	&\Wnum_{T, I}
	 \colon \R^{ I }  \to \R^{ J }
	 \colon \fnum  \mapsto (  \Wnum_T  \fnum )_{I^c}
	 \\
	 &\Wnum_{T, I}^\sharp
	 \colon \R^{ J } \to \R^{ I }
	 \colon \gnum \mapsto ( \Pnum_I    \Rnum_T   \Enum_I ) \gnum  \,. 
\end{align}
Here $ \Wnum_T  \colon \R^{N \times N} \to \R^{N \times N}$ and $ \Rnum_T  \colon \R^{N \times N} \to \R^{N \times N}$  are discrete analogs  of the forward wave equation  and its time reversed version, $\Enum_I$ a discretization of the harmonic extension  operator and $\Pnum_I$ a discretization of the projection  of 
the projection onto $H^1_0(\Om)$.        

The   numerical solution  of the wave equation $\Wnum_T \fnum$ and likewise  the numerical solution of the time reversed version $\Rnum_T$ are computed with the $k$-space method  \cite{bojarski1982k,cox2007k,mast2001k}.   We use the  $k$-space method  with periodic boundary conditions on the rectangle  $[-a,a]^2$ as described in  \cite{haltmeier2017analysis}. We choose $a \geq  T + 1$  such that  $\Wnum_{T, I}$ and   $\Wnum_{T, I}^\sharp $ are not  affected  by replacing the free space wave equation with its $(2a)$-periodic counterpart. The discrete harmonic extension $\Enum_I$ and  the discrete projection $\Pnum_I$  are constructed by numerically solving \eqref{eq:laplace} with the MATLAB-routine {\tt solvepde}.   

\begin{figure}[htb!]
\centering
\includegraphics[width=0.32\textwidth]{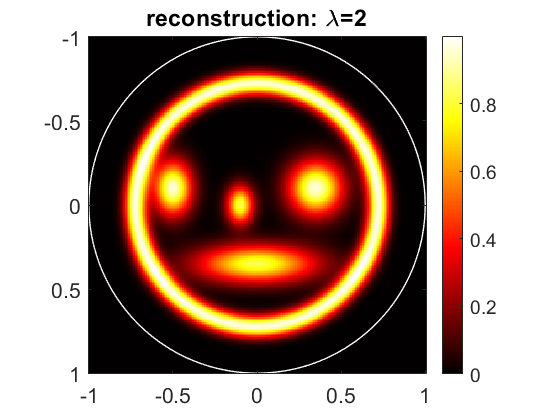}
\includegraphics[width=0.32\textwidth]{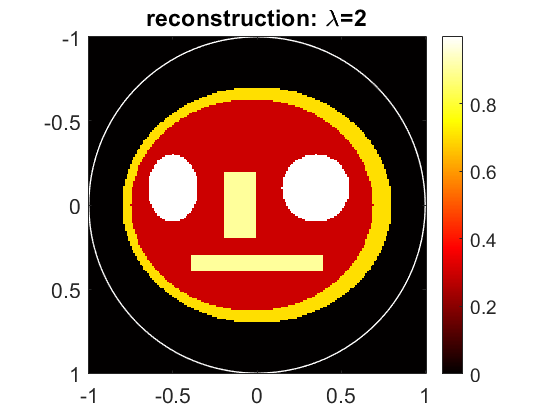}
\includegraphics[width=0.32\textwidth]{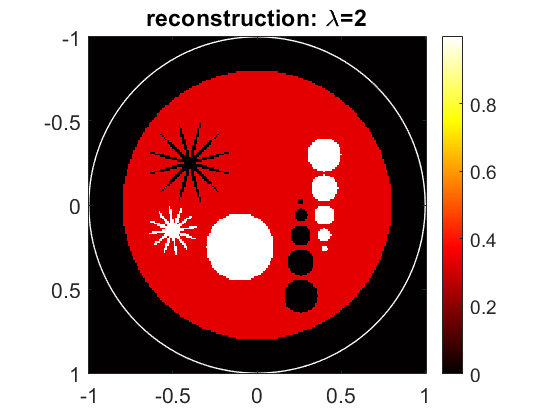} \\[0.2em]
\includegraphics[width=0.32\textwidth]{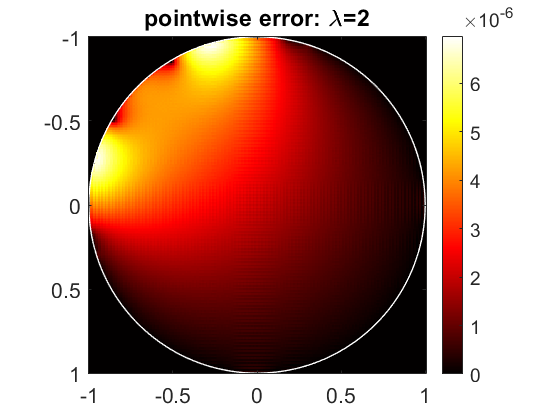}
\includegraphics[width=0.32\textwidth]{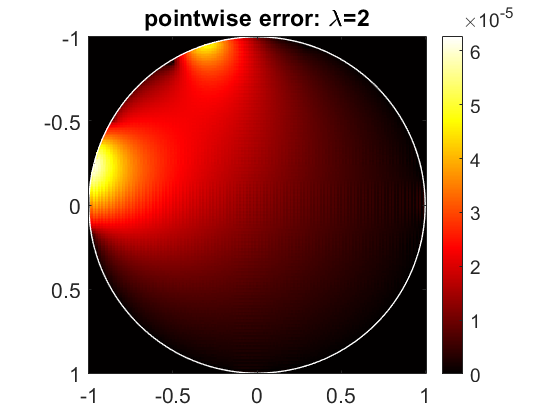}
\includegraphics[width=0.32\textwidth]{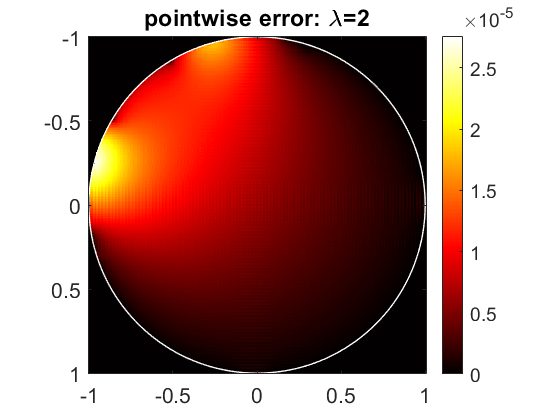} 
\caption{\textbf{Results for constant sound speed.} Reconstructions (top) and corresponding point-wise errors (bottom) using constant sound speed  $c_{\rm I}=1$.}
\label{fig:constant}
\end{figure}

\begin{figure}[htb!]
\centering
\includegraphics[width=0.32\textwidth]{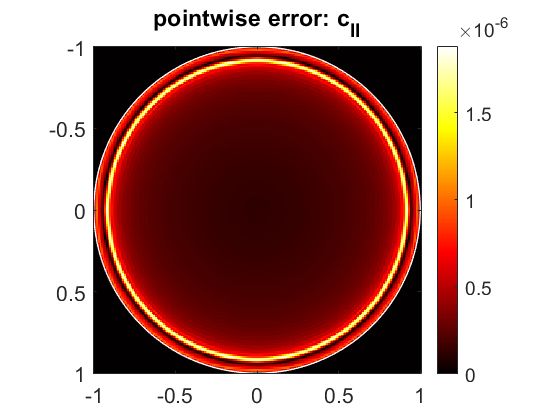}
\includegraphics[width=0.32\textwidth]{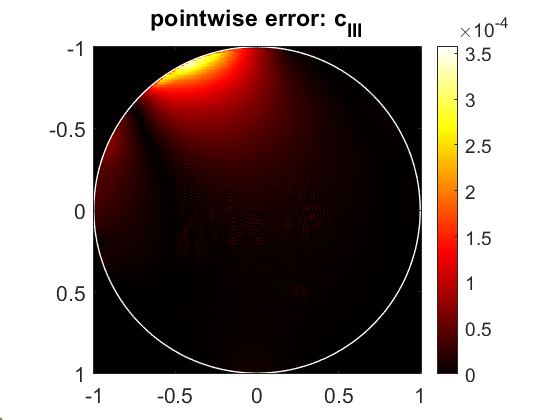}
\includegraphics[width=0.32\textwidth]{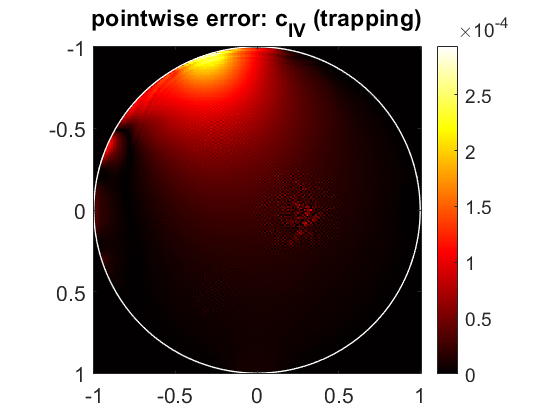} 
\caption{\textbf{Results for variable sound speed.} Point-wise errors (bottom) using using different sound speed profiles  the smooth phantom $f^a$.}
\label{fig:variable}
\end{figure}

\subsection{Numerical results}
We first present results to data $\gnum = \Wnum_T \fnum$ without added noise.  Figure~\ref{fig:constant} shows results with constant speed with relaxation parameter $\la =2$ and $80$ iterations. For smaller values of  $\la$ slightly better reconstructions have been obtained but required a slightly larger number  of iterations. Figure~\ref{fig:variable}
visualizes the pointwise error map $f^a - f^a_{\rm rec}$ for the non-constant sound speed profiles using  $\la=1/2$ and 
$T=2$. We see that accurate results are obtained for all sound speed profiles. The best results were obtained  for the sound speed profile $c_{\rm II}$,  and the error functions do not contain any visible information of the original phantom.  Because all reconstructions look equally well and very similar to the original phantom $f^a$, we did not visualize them here.  Additional simulations with other smooth and non-smooth phantoms indicate that smooth phantoms generally result in  better reconstructed than non-smooth ones.  
\begin{figure}[htb!]
\centering
\includegraphics[width=4.5cm,angle=0]{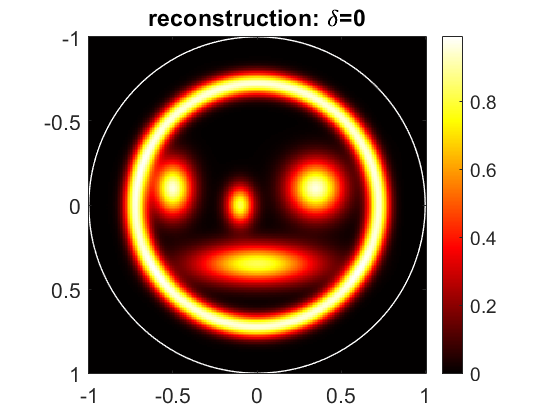}
\includegraphics[width=4.5cm,angle=0]{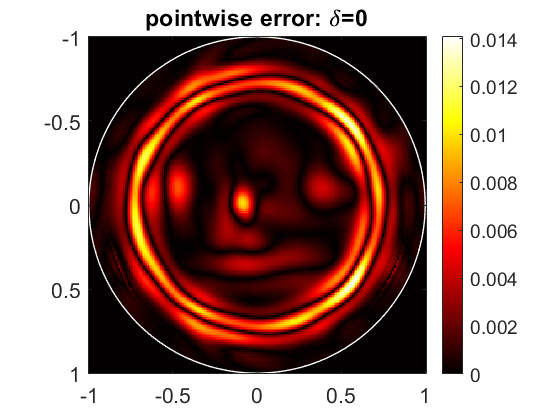}
\includegraphics[width=4.5cm,angle=0]{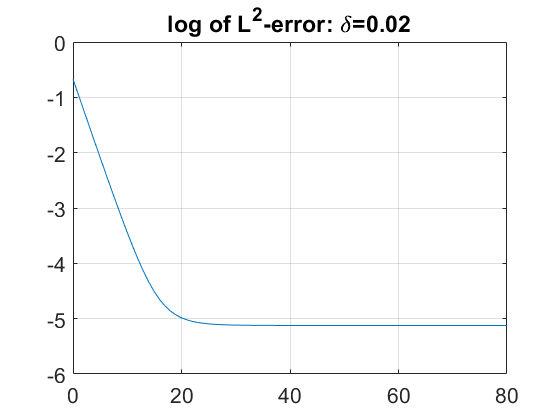}\\
\includegraphics[width=4.5cm,angle=0]{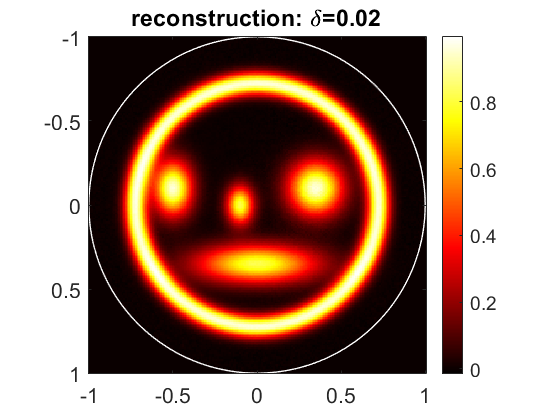}
\includegraphics[width=4.5cm,angle=0]{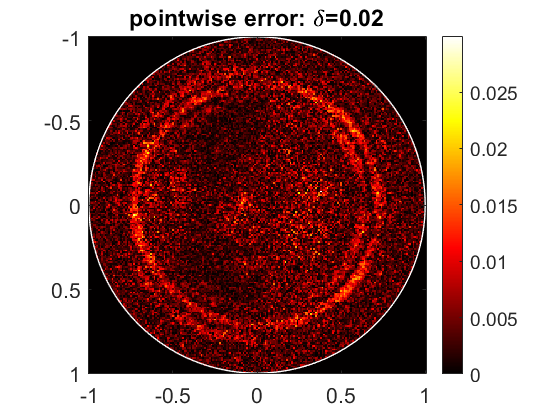}
\includegraphics[width=4.5cm,angle=0]{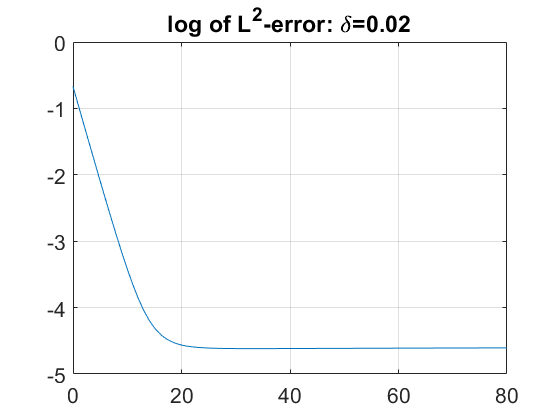}
\caption{\textbf{Exact versus noisy data for  sound speed $c_{\rm III}$.}  From left to right:  reconstruction, difference images to true phantom $f^a$, and logarithmic error plot in dependence of the number of iterations.   The top row shows results for exact data, the bottom row shows results for noisy data. Here  $\la = 1/2$ and $T=4$.}
\label{fig:non-trapping}
\end{figure}

\begin{figure}[htb!]
\centering
\includegraphics[width=4.5cm,angle=0]{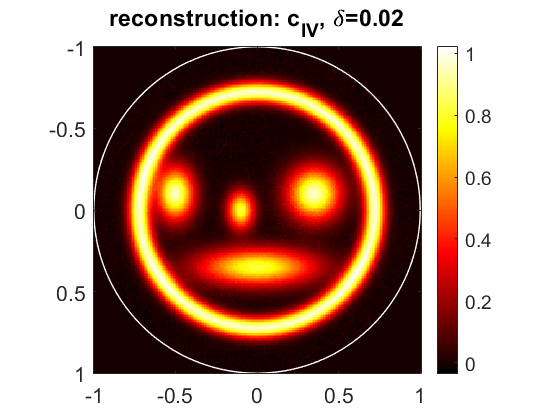}
\includegraphics[width=4.5cm,angle=0]{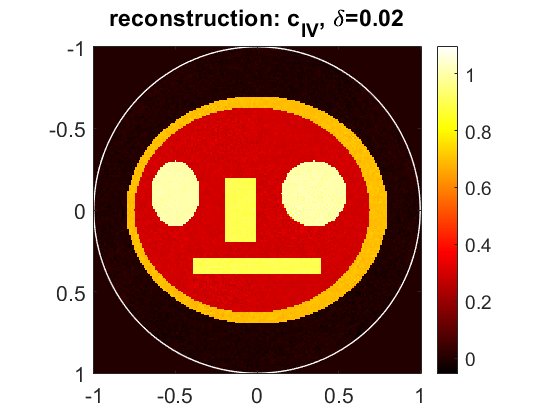}
\includegraphics[width=4.5cm,angle=0]{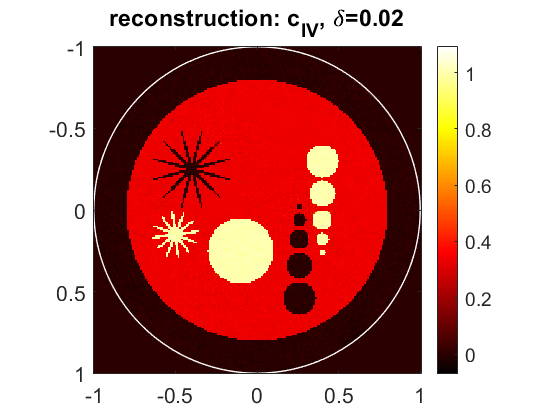}\\
\includegraphics[width=4.5cm,angle=0]{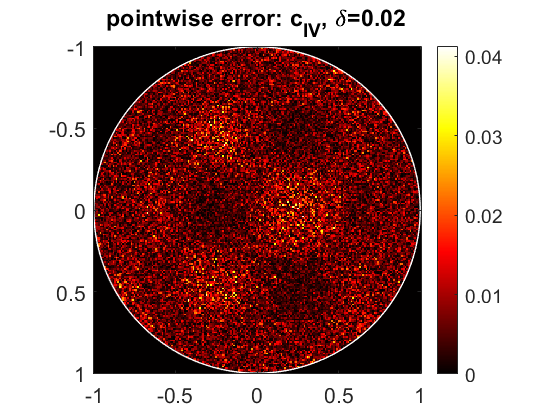}
\includegraphics[width=4.5cm,angle=0]{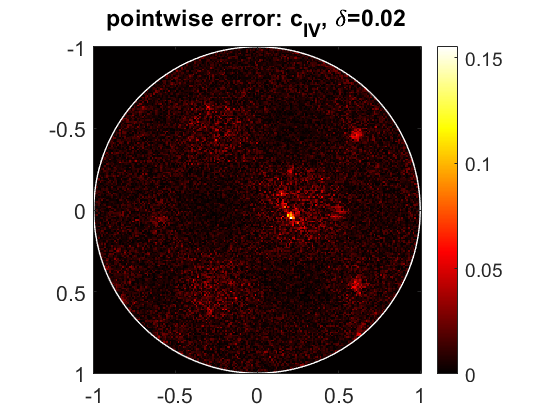}
\includegraphics[width=4.5cm,angle=0]{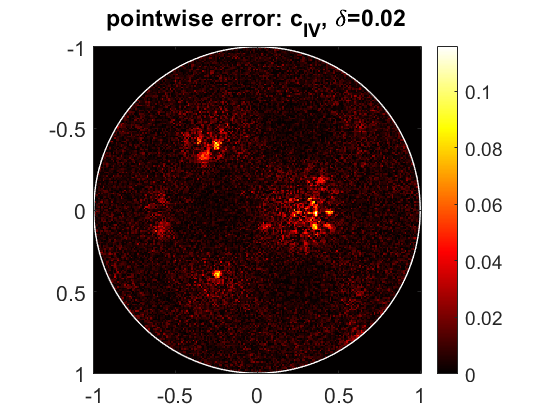}
\caption{\textbf{Results for noisy data for trapping sound speed $c_{\rm IV}$.}  Top row shows the  reconstructions of the phantom $f^a$, $f^b$  and $f^c$. Bottom row: Corresponding difference images for to the true phantom. Here  $\la = 1/2$ and $T=2$}
\label{fig:trapping}
\end{figure}

Next we present results for noisy data where $\Wnum_T^\sharp  \fnum$ has been contaminated with normally distributed noise with a standard deviation of  two percent of the maximal pressure value.  In order to avoid inverse crime, data are simulated using a three times finer discretization than used for the reconstruction.
We observe that in all cases the iteration process is  stable when $\la$ is  chosen sufficiently small and the use of a stopping rule is not necessary. Moreover, the reconstructions are more accurate for smooth phantoms than for  piecewise constant phantoms.   Finally in Figure \ref{fig:trapping} we show results for the trapping sound speed $c_{\rm iv}$. Also in this case, the iterative time-reversal works well for all phantoms even the  theory developed in the previous Sections does not fully apply in this situation.

\section{Conclusion}
\label{sec:conclusion}

In this work we studied an inverse source problem appearing in full field PAT.  Image reconstruction amounts to the  inversion of the exterior final time wave transform $\Wo_{T, \Omega}$ that maps the initial data $f$ supported in $\Omega$  to the solution of the wave equation at fixed time $T$ restricted  to the complement $\Omega^c$.    For non-constant sound speed, besides the work \cite{haltmeier2019photoacoustic}, to the best of our knowledge, inversion of $\Wo_{T, \Omega}$ is studied for the first time.   We, for the first time, derived uniqueness  and stability results. Moreover we  showed convergence of the proposed iterative time-reversal reconstruction algorithm. For that purpose we have proven that $\Io - \la \Wo_{T, \Omega}^\sharp \Wo_{T, \Omega} $ is a contraction on $H^1_0(\Omega)$ for all $\la \in  (0,2)$ where $\Wo_{T, \Omega}^\sharp $  is a modified time-reversal  operator. We also derived a numerical realization of the iterative time-reversal algorithm. Numerical results show accurate reconstruction for all sound speed profiles and all initial data.

\section*{Acknowledgments}
M.H.  acknowledges support of the Austrian Science Fund (FWF), project P 30747-N32.
The~research of L.N. has been supported by the National Science Foundation (NSF) Grants DMS 1212125 and DMS 1616904.

\end{document}